\documentclass{article}
\usepackage{amsmath,amssymb,amsthm,stmaryrd,mathrsfs}

\usepackage{color,graphicx}
\usepackage[dvipsnames]{xcolor}
\usepackage{accents}
\usepackage{enumerate}
\usepackage{varwidth}
\usepackage{enumitem}

\usepackage{fancyhdr, array,calc,graphicx,url,tabularx}
\usepackage{geometry}
\usepackage{latexsym}
\usepackage{amssymb}
\usepackage{amsmath}
\usepackage{enumerate}
\usepackage{verbatim}
\usepackage{tikz}
\usepackage[colorlinks=true,linkcolor=blue]{hyperref}

\usepackage{changepage}

\newtheorem{theorem}{Theorem}[section]

\newtheorem{definition}[theorem]{Definition}
\newtheorem{lemma}[theorem]{Lemma}
\newtheorem{corollary}[theorem]{Corollary}

\newcommand{\R}{\mathbb{R}}
\newcommand{\Q}{\mathbb{Q}}
\newcommand{\ADR}{\mathsf{AD}_{\mathbb{R}}}

\newcommand{\AD}{\mathsf{AD}}
\newcommand{\AC}{\mathsf{AC}}
\newcommand{\DC}{\mathsf{DC}}

\newcommand{\pmax}{\mathbb{P}_{\mathrm{max}}}

\begin{document}

\title{The Cofinal Strong Chang Conjecture from Models of Determinacy}
\author{Lagadec Corentin \thanks{The author’s work is funded by the National Science Center, Poland (project number 2023/50/A/ST1/00258) and partially supported by the Simons Foundation grant (award no. SFI-MPS-T-Institutes-00010825) and from State Treasury funds as part of a task commissioned by the Minister of Science and Higher Education under the project “Organization of the Simons Semesters at the Banach Center - New Energies in 2026-2028” (agreement no. MNiSW/2025/DAP/491).}}

\maketitle

\begin{abstract}
    In chapter 9 of [6], Woodin shows how to force the Strong Chang Conjecture over models of determinacy using $\pmax$. We show here how a modification of the proof implies that such extensions actually verify the stronger cofinal version of the conjecture. This stronger version has important consequences on the semi-properness of small forcing, allowing us to prove the consistency of the theory "ZFC + Namba forcing is semiproper + $\Theta^{UB}=\omega_3$". We then use the constructions of this proof to also show that Woodin $(\ast)_{UB}$ axiom implies the conjecture.
\end{abstract}

\section{Chang Conjecture}

In the rest of this paper the models considered will always verify "ZF + DC", this will be enough to make sense of countable elementary submodels which we need for the Strong Chang Conjecture. We start with some combinatorial properties of elementary submodels needed to handle the basics of the conjecture, we refer to [1] for the details.

\begin{definition}
    Let $M\prec H_\theta$, for any $Y\in H_\theta$ pose $M[Y]=\{f(x)\ |\ f\in M, x\in Y^{<\omega}\cap dom(f) \}$ and $M(\beta)=M[\{\beta \}]$. 
\end{definition}

\begin{lemma} \label{esm}
    If $\DC_{<\theta}$ holds, then $M[Y]\prec H_\theta$ and for any $N\prec H_\theta$, if $M,Y\subseteq N$ then $M[Y]\subseteq N$.
\end{lemma}

Although the next result is often used in a ZFC context we remark that it is always possible to just force choice over the models we use without adding countable sequences. In this situation if a set is (projective) stationary in a generic extension  then this set had to be (projective) stationary to begin with in the ground model.

\begin{theorem} \label{projective stationary}
    For any $\delta \geq \omega_2 $ and any structure U on $H_{\delta}$, the set \{$M\in [H_{\delta}]^\omega$ $|\ M\prec U$ and for cofinally many $\alpha\in \omega_2 - M$, $M(\alpha)\cap \omega_1 = M\cap \omega_1$\} is projective stationary.
\end{theorem}

Here we introduce some weak form of the Chang Conjecture used in the proof. The idea of the proof is essentially that in suitable generic extensions of models of determinacy the following weak versions of the Chang Conjecture reflects to the strong versions. 

\begin{definition}
    $WCC^{+}$ : For any $F : \omega_2^{<\omega} \rightarrow \omega_2 $, there is  $G : \omega_2^{<\omega} \rightarrow \omega_2 $ such that for any $X\in [\omega_2]^{\omega}$, if X is closed by G then there is $Y\in [\omega_2]^{\omega}$ such that :
    \begin{enumerate}
    \item X $\subsetneq$ Y.
    \item $X \cap \omega_1 = Y \cap \omega_1$.
    \item Y is closed by F.
    \end{enumerate}
\end{definition}

\begin{definition}
    $WCC^{+}_{cof}$ : For any $F : \omega_2^{<\omega} \rightarrow \omega_2 $, there is  $G : \omega_2^{<\omega} \rightarrow \omega_2 $ such that for any $X\in [\omega_2]^{\omega}$, if X is closed by G then for any $\alpha < \omega_2$ there is $Y\in [\omega_2]^{\omega}$ such that :
    \begin{enumerate}
    \item X $\subsetneq$ Y.
    \item sup Y $> \alpha$.
    \item $X \cap \omega_1 = Y \cap \omega_1$.
    \item Y is closed by F.
    \end{enumerate}
\end{definition}

Here are the stronger versions of the conjecture we are interested into proving.

\begin{definition}
    Strong Chang Conjecture (SCC) : there is club many countable $M \prec H_{\omega_3}$ such that there is a countable $N \prec H_{\omega_3}$ such that : 
    \begin{enumerate}
    \item M $\subseteq$ N.
    \item $M\cap \omega_2 \neq N\cap \omega_2$.
    \item $N \cap \omega_1 = M \cap \omega_1$.
    \end{enumerate}
\end{definition}

\begin{definition}
    Strong Chang Conjecture cofinal $(SCC^{cof})$ : there is club many countable $M \prec H_{\omega_3}$ such that for cofinally many $\alpha$ in $\omega_2$ there is a countable $N \prec H_{\omega_3}$ such that : 
    \begin{enumerate}
    \item M $\subseteq$ N.
    \item $M\cap \omega_2 \neq N\cap \omega_2$.
    \item sup $N\cap\omega_2 > \alpha$.
    \item $N \cap \omega_1 = M \cap \omega_1$.
    \end{enumerate}
\end{definition}

This next result show that by lifting to higher $H_\theta$ we can obtain that ANY submodel can be end-extended.

\begin{theorem} $(\AC)$
    Let $\theta$ be a cardinal such that $H_{\omega_3}\in H_\theta$ and $\lhd$ be a well ordering of $H_\theta$, then SCC if equivalent to :
    For ANY countable $M \prec \langle H_\theta,\in,\lhd \rangle$, there is a countable $N \prec \langle H_\theta,\in,\lhd \rangle$ such that :
    \begin{enumerate}
    \item M $\subseteq$ N.
    \item $M\cap \omega_2 \neq N\cap \omega_2$.
    \item $N \cap \omega_1 = M \cap \omega_1$.
    \end{enumerate}
\end{theorem}

The same is true for the cofinal version. 
\\
\\
The next result introduce an equivalent statement that will be easier to prove in the main theorem.

\begin{lemma} $(\omega_2-\DC)$
    Let $M_{0}=\{f : \omega_2^{<\omega} \rightarrow \omega_2\}\cup H_{\omega_2}$.
    Suppose that there are club many $X\in [M]^{\omega}$ such that there is $Y\in [\omega_2]^{\omega}$ such that : 
    \begin{enumerate}
        \item $X\cap \omega_2 \subsetneq Y$.
        \item $X\cap \omega_1 = Y\cap \omega_1$.
        \item for every $f\in X, f : \omega_2^{<\omega} \rightarrow \omega_2$, Y is closed by f.

    \end{enumerate}
    Then SCC holds.
\end{lemma}

\begin{proof}
    Set $C_{0}$ the club of all X in $[M_{0}]^{\omega}$ for which there is Y that verify (1)-(3) and $C=\{M\in [H_{\omega_3}]^{\omega}\ |\ M\prec H_{\omega_3}\ and\ M\cap M_{0}\in C_{0}\}$ . We show that C witnesses SCC. 
    C is obviously a club.
    Recall that $M[Y]=\{f(x)\ |\ f\in M, x\in Y^{<\omega}\cap dom(f)\}$ and $M\prec M[Y]\prec H_{\omega_3}$.
    For any M in C, pick Y that witnesses $M\cap M_{0}\in C_{0}$, for N=M[Y] we have $N\cap \omega_2 = Y$ and thus :
    \begin{enumerate}
        \item $M \subseteq N$.
        \item $M\cap \omega_2 = X\cap \omega_2 \neq N\cap \omega_2 = Y $.
        \item $M\cap \omega_1 = X\cap \omega_1 = N\cap \omega_1 = Y \cap \omega_1 $.
    \end{enumerate}
\end{proof}

Adding a cofinal parameter uses the same proof :

\begin{lemma}$(\omega_2-\DC)$ \label{scc}
    Let $M_{0}=\{f : \omega_2^{<\omega} \rightarrow \omega_2\}\cup H_{\omega_2}$.
    Suppose that there are club many $X\in [M]^{\omega}$ such that for cofinally many $\alpha \in \omega_2$ there is $Y\in [\omega_2]^{\omega}$ such that : 
    \begin{enumerate}
        \item $X\cap \omega_2 \subsetneq Y$.
        \item $sup\ Y > \alpha$.
        \item $X\cap \omega_1 = Y\cap \omega_1$.
        \item for every $f\in X, f : \omega_2^{<\omega} \rightarrow \omega_2$, Y is closed by f.

    \end{enumerate}
    Then $SCC^{cof}$ holds.
\end{lemma}

Since the models we use are of the form $L(\Gamma,\R)$ or $HOD(\Gamma)$, $\pmax$ does not force full choice, the following allow us to keep the conjecture after forcing $\AC$. In those models forcing a well order of $P(\R)$ (with $coll(\omega_{3},P(\R))$ for example) will force $\AC$. Again, the same is true for the cofinal version.

\begin{theorem}
    Suppose N is an inner model such that $N \models ZF + \DC + SCC$ and $P(\omega_2)\subseteq N$, then SCC holds.
\end{theorem}

\begin{corollary}
    If $V\models SCC$, then SCC holds in any generic extension by an $\omega_2$-closed forcing.
\end{corollary}

As we will see that $\pmax$ forces $WRP(\omega_2)$ over models of $\Theta_{reg}$, we use the following to obtain the weak version of the conjecture in our model. The original proof uses the following.

\begin{lemma}
    $WRP(\omega_2)\ implies\ WCC^+$.
\end{lemma}

\begin{proof}
    Suppose there is $f : \omega_2^{<\omega} \rightarrow \omega_2$ such that the set S of sets for which $WCC^+$ fails for $f$ is stationary. By $WRP(\omega_2)$ there is $\lambda \in \omega_2, \lambda > \omega_1$ such that $S\cap [\lambda]^\omega$ is stationary in $[\lambda]^\omega$. Pose $S_{0}=\{X  \in [\omega_2]^{\omega}\ |\ X\cap \lambda \in S\}$, $S_{0}$ is stationary so there is $M\prec H_{\omega_2}$ such that $M\cap \omega_2 \in S_{0}$ and $f\in M$. But then $M\cap \omega_2$ is closed by $f$ and witnesses $M \cap \lambda \notin S$.
\end{proof}

The first improvement over [6] comes from the fact that $WRP(\omega_2)$ actually implies the stronger $WCC^+_{cof}$:

\begin{lemma}
    $WRP(\omega_2)$ implies $WCC^{+}_{cof}$.
\end{lemma}

\begin{proof}
    Fix $f : \omega_2^{<\omega} \rightarrow \omega_2$ and suppose that the set S of sets for which $WCC^{+}_{cof}\ fails\ at\ f$ is stationary. we have $S=\{X\in [\omega_2]^{\omega}\ |\ \exists \alpha \in \omega_2\ \forall Y \in [\omega_2]^{\omega},\ if\ Y\ verify\ conditions\ (1)-(3)\ of\ WCC^{+}_{cof},\ then\ Y\ fails\ (4) \}$.
    By $WRP(\omega_2)$ there is $\lambda \in \omega_2$ such $S\cap [\lambda]^\omega$ is stationary in $[\lambda]^\omega$.
    \\
    \\
    Fix a surjection $g : \lambda \rightarrow \omega_1$, the set $T=\{\alpha\in \omega_1\ |\ g[\alpha] \in S \cap [\lambda]^\omega \}$ is stationary (for $\omega_1$). By lemma \ref{projective stationary} there is $M\prec (H_{\omega_2},f,g,\in)$ such that 
    \begin{enumerate}
        \item $M\cap \omega_1\in T$.
        \item for cofinally many $\beta \in \omega_2$, $M(\beta)\cap\omega_1=M\cap\omega_1$.
    \end{enumerate}
    Since $g\in M$, we have for $X=g[M\cap \omega_1]$ : 
    \begin{enumerate}
        \item $X\cap\omega_1 = M\cap\omega_1$.
        \item $X\subseteq M$.
    \end{enumerate}
    because $X\cap\omega_1 \in T$ we also have $X\in S$, fix $\alpha$ such that $WCC^{+}_{cof}\ fails\ above\ \alpha$ for X. For any $\beta > \alpha$ such that $M(\beta)\cap \omega_1 = M\cap \omega_1$ and Y=$M(\beta)\cap \omega_2$ :
    \begin{enumerate}
        \item Y is closed by $f$.
        \item $Y\cap \omega_1 = M(\beta)\cap\omega_1 = M\cap\omega_1 = X\cap \omega_1$.
        \item $X \subsetneq Y$ since $\beta \in Y-X$.
        \item sup Y  $>\alpha$.
    \end{enumerate}
    This contradict the assumption $X\in S$.
\end{proof}

\section{$\pmax$, Forcing and Determinacy}

\subsection{$\pmax$ forcing}

We do not need to go into the internal of $\pmax$, using $\pmax$ only make sense in models of $\AD^+$ which will be the case here. The following properties of $\pmax$ are sufficient to get the result we want : $\pmax$ is $\omega$-closed (and thus does not add reals or countable sequences), weakly-homogeneous, $\pmax\subseteq H_{\omega_1}$ and being a $\pmax$ condition is $\Pi^1_2$, implying that the structures we encounter are correct enough about $\pmax$. We will also use the following results :

\begin{theorem} \label{wrp}
    If $V=L(P(\mathbb{R}))$, $\AD_{\mathbb{R}}$ holds and $\Theta$ is regular, then $\pmax$ forces $WRP(\omega_2)$.
\end{theorem}

\begin{theorem}
    Suppose $\Gamma$ is closed under continuous preimages and $L(\Gamma,\R)\models \AD^+$, then if $G\subseteq \pmax$ is $L(\Gamma,\R)$-generic :
    $L(\Gamma,\R)[G]\models \omega_2^V$=$\omega_2^{V[G]}$, $\Theta^{V} >\omega_2^{V[G]}$ and $|\R|=\omega_2$.
\end{theorem}

\begin{theorem}
     Suppose $\Gamma$ is closed under continuous preimages and $L(\Gamma,\R)\models \AD^+$, then if $G\subseteq \pmax$ is $L(\Gamma,\R)$-generic the following are equivalent :
     \begin{enumerate}
         \item $L(\Gamma,\R)[G]\models \omega_2-\DC$.
        \item $\Theta$ is regular.
     \end{enumerate}
\end{theorem}

We see that the properties used here are very general, this could indicate that other forcing notions might also force the conjecture over certain models of determinacy. An interesting point is that SCC is a statement about $H_{\omega_3}$ that can be forced with $\pmax$ while most of the usual consequences of $\pmax$ are focused on $H_{\omega_2}$. 

\subsection{Determinacy}

Here is a list of determinacy related results used in the proof. 

\begin{theorem}
    Suppose $\ADR$ holds, then every set $A\subseteq \R$ is $<\Theta$-homogeneous Suslin.
\end{theorem}

\begin{theorem} (Sargsyan)
    Suppose $\ADR$ holds, then every set of real is $\Theta$-universally Baire.
\end{theorem}

As a corollary we have :

\begin{theorem} \label{ub}
    Suppose $\ADR$ holds and let $(\theta_\beta, \beta<\Omega)$ be the $\Theta$-sequence. For any $\alpha<\Omega$ and $\beta>\alpha$ such that $HOD\models$"$ \theta_\beta$ is regular", $V_{\theta_\beta}^{HOD(\Gamma)}\models$ "$ZF + \AD^+ +$ all sets of reals are universally baire " where $\Gamma=\{A\subseteq \R\ |\ w(A)<\alpha \}$.
\end{theorem}

Since $\pmax$ forces $\Theta$ to become $\omega_3$ when $\Theta$ is regular this can be use to obtain a model where $\Theta^{UB}=\omega_3$.

\begin{theorem}
    Suppose $\ADR$ and $\Theta$ is regular, then $\AD^+$ holds.
\end{theorem}

Note that the definition of $\AD^+$ includes $\DC$. The next results are classic $\AD^+$ consequences.  Since the models of determinacy we consider will all be models of $\ADR$, in them all sets of reals will be suslin/co-suslin. The next vital result will be used to get ordinal definable sets from the strategies we get from $\ADR$.

\begin{theorem} \label{SBT}
    (Solovay basis theorem)
    Assume $\AD^+$ and let A be a Suslin co-Suslin set of reals, then any $\Sigma^2_1(A)$ non empty collection of sets of reals has a $\Delta^2_1(A)$ element.
\end{theorem}

\begin{theorem} \label{coding}
    Suppose $\Gamma$ is a pointclass closed under continuous pre-images such that\\
    \\
    \centerline{$L(\Gamma,\R)\models \AD^{+}$.}
    \\
    \item Let $\langle \Theta_\alpha : \alpha < \Omega \rangle$ be the $\Theta-sequence$ of $L(\Gamma,\R)$. Suppose that $\Omega$ is a limit ordinal,
    then there is a surjection $\pi : \Theta^\omega\cap V_\Theta \rightarrow P(\R)\cap L(\Gamma,\R)$ that is $\Sigma_1$-definable in $L(\Gamma,\R)$ from $\{\R\}$.
\end{theorem}

\begin{definition}
    For a given partial order $\mathbb{P}$, define a topology on the set of all filters $G\subseteq \mathbb{P}$ with basic open sets $N_p = \{G\ |\ p\in G\}$. In this topology a set E is dense if for any $p\in\mathbb{P}$ there is $G\in E$ such that $p\in G$ and is open if for any $G\in E$ there is $p\in G$ such that $N_p \subseteq E$.
\end{definition}

\begin{theorem} \label{comeager}
    Suppose $\Gamma$ is a pointclass closed under continuous pre-images such that\\
    \\
    \centerline{$L(\Gamma,\R)\models \AD^{+}$.}
    \\
    \\
    Let $\Theta = \Theta^{L(\Gamma,\R)}$ and suppose that $a\in P_{\omega_1}(\Theta^\omega)\cap V_\Theta$, then for $\mathbb{P}\in HOD^{L(\Gamma,\R)}(a)$ countable in V and X a comeager set of filters of $\mathbb{P}$ that is ordinal definable in $L(\Gamma,\R)$ with parameter a $\cup$ \{a\}, then any $HOD^{L(\Gamma,\R)}(a)$-generic filter of $\mathbb{P}$ is in X.

\end{theorem}

\begin{theorem}
    Note $\mu$ the club filter of $P_{\omega_1}(\R)$.
    Suppose $\ADR$ holds and $\Theta$ is regular, then $\mu$ is a countably complete ultrafilter.
\end{theorem}

From this theorem we pose the following :

\begin{definition}
    In the context of the previous theorem we define an embedding\\
    \centerline{$j_\mu : \cup\{L[S]\ |\ S\subset Ord\} \rightarrow V$}
    where $j_\mu | L[S]$ is the embedding associated to the ultrapower $\{f : P_{\omega_1}(\R) \rightarrow L[S]\ |\ f\in V\}/\mu$.
\end{definition}

Note that this embedding is not elementary, we are only interested in the fact that it is ordinal definable (since $\mu$ is definable) and its relation with the next result.

\begin{theorem} \label{coll embedding}
    Suppose $\Gamma$ is a pointclass closed under continuous pre-images such that
    \\
    \\
    \centerline{$L(\Gamma,\R)\models \ADR + \Theta\ is\ regular$}
    \\
    \item Suppose $G\subseteq Coll(\omega,\R)$ is $L(\Gamma,\R)$-generic, then there is a generic elementary embedding\\
    $j_G : L(\Gamma,\R) \rightarrow N \subseteq L(\Gamma,\R)[G]$
    such that :
    \begin{enumerate}
        \item $N^\omega \subseteq N$ in $L(\Gamma,\R)[G]$.
        \item for any set of ordinals S in $L(\Gamma,\R)$,
    $j_G | L[S] = j_\mu |L  [S]$.
    \end{enumerate}
    
\end{theorem}

In particular, if S is a set of ordinals then $j_G(S)\in V$ and is OD[S].
\\
\\
We also note that by the universal property of the levy collapse any forcing of size $\R$ embeds into $coll(\omega,\R)$, as a result any extension by $coll(\omega,\R)$ will contain a $\pmax$ filter.
\\
\\
Recall that $\Q_{<\delta}$ is the stationary tower forcing below $\delta$, we only use it in a ZFC context. The conditions of $\Q_{<\delta}$ are stationary sets on  $[X]^\omega$ for some $X\in V_\delta$. 

\begin{theorem}
    If $G\subseteq \Q_{<\delta}$ is V-generic and $\delta$ is Woodin, then there is a generic elementary embedding $j_G : V \rightarrow N \subseteq V[G]$ such that :
    \begin{enumerate}
        \item $crit(j_G)=\omega_1$.
        \item $j_G(\omega_1)=\delta$.
        \item $N^{\omega}\subseteq N$ in $V[G]$.
    \end{enumerate}
\end{theorem}

\begin{theorem}
    Suppose $Y\in V_\delta$, then for any $G\subseteq \Q_{<\delta}$ V-generic and $S\subseteq [Y]^\omega$ stationary, then $S\in G$ iff $j_G[Y]\in j_G(S)$.
\end{theorem}

\begin{theorem} \label{tower}
    Let $\delta$ be a Woodin cardinal, A $\subseteq \R$ be a $\kappa$-homogeneous set with $\kappa > \delta$, $G\subseteq \Q_{<\delta}$ be generic and $j_G$ the associated generic embedding, then $j_G (A)=A^G$ where $A^G$ is the canonical interpretation of A in $V[G]$.
\end{theorem}

We also introduce a generalized version of the stationary tower at $\omega_2$ : $\Q_{<\lambda}^{\omega_2}$. Conditions are stationary sets on $[X]^{\omega_2}$ for some $X\in V_\lambda$.

\begin{theorem}
    If $H\subseteq \Q_{<\lambda}^{\omega_2}$ is V-generic, $\lambda$ is Woodin and the set $\{X\in [\omega_3]^{\omega_2}\ |\ X\cap\omega_3\in\omega_3\}$ is in H, then there is a generic elementary embedding $j_H : V \rightarrow N \subseteq V[H]$ such that :
    \begin{enumerate}
        \item $crit(j_H)=\omega_3$.
        \item $j_H(\omega_1)=\lambda$.
        \item $N^{\omega_2}\subseteq N$ in $V[H]$.
    \end{enumerate}
\end{theorem}

\section{The Strong Chang Conjecture in $\pmax$ models}

We improve the following theorem from Woodin.

\begin{theorem}
    Suppose $\Gamma$ is a pointclass closed under continuous pre-images such that
    \\
    \\
    \centerline{$L(\Gamma,\R)\models \ADR + \Theta\ is\ regular$}
    \\
    \item Let $\langle \Theta_\alpha : \alpha < \Omega \rangle$ be the $\Theta-sequence$ of $L(\Gamma,\R)$, pose for each $\delta<\Omega$,
    
    \centerline{$\Gamma_\delta = \{A\subseteq \R\ |\ w(A)< \Theta_\delta\}$}
    
    \item set
    
    \centerline{$N_{\delta}=HOD(\Gamma_\delta)$}
    
    \item if $\delta<\Omega$ is such that 
    \begin{enumerate}
        \item $\delta=\Theta_\delta$.
        \item $N_\delta\models \delta\ is\ regular$.
    \end{enumerate}
    then for $G_0 \subseteq \pmax$ $L(\Gamma,\R)$-generic :
     $N_\delta [G_0]\models ZF+\omega_2-\DC+SCC$.
\end{theorem}

Note that the existence of such a $\delta$ is not implied by $\ADR + \Theta$ regular and requires slightly stronger hypotheses. For example $\ADR + \Theta$ is Mahlo suffices. Another remark is that we do not force the Strong Chang Conjecture over a model of the form $L(\Gamma,\R)$ but over a model of the form $HOD(\Gamma_\delta)$. 
\\

We show that in the conclusion SCC can be strengthened to $SCC^{cof}$, the proof is the same where we start with $WRP(\omega_2)$ implies $WCC^+_{cof}$ instead of just $WCC^+$ and show how the parameter $\alpha$ is carried through it.

\begin{theorem} \label{scc pmax}
    Suppose $\Gamma$ is a pointclass closed under continuous pre-images such that
    \\
    \\
    \centerline{$L(\Gamma,\R)\models \ADR + \Theta\ is\ regular$}
    \\
    \item Let $\langle \Theta_\alpha : \alpha < \Omega \rangle$ be the $\Theta-sequence$ of $L(\Gamma,\R)$, pose for each $\delta<\Omega$,
    
    \centerline{$\Gamma_\delta = \{A\subseteq \R\ |\ w(A)< \Theta_\delta\}$}
    
    \item set
    
    \centerline{$N_{\delta}=HOD(\Gamma_\delta)$}
    
    \item if $\delta<\Omega$ is such that 
    \begin{enumerate}
        \item $\delta=\Theta_\delta$
        \item $N_\delta\models \delta\ is\ regular$
    \end{enumerate}
    then for $G_0 \subseteq \pmax$ $L(\Gamma,\R)$-generic :
     $N_\delta [G_0]\models ZF+\omega_2-\DC+SCC^{cof}$.
\end{theorem}

\begin{proof}
    We begin by fixing $G_\infty \subseteq Coll(\omega,\R)$ $L(\Gamma,\R)$-generic, $j_\infty$ be the derived embedding given by theorem \ref{coll embedding}. Fix also $G_0 \subseteq \pmax$ $L(\Gamma,\R)$-generic such that $G_0\in L(\Gamma,\R)[G_\infty]$.
    \\
    \\
    In the rest of the proof $\R$ will always refer to the "original" set of reals of the base model, the reals of any other model will always be noted differently (such as $\R^\infty$).
    \\
    \\
    By theorem \ref{wrp} $WRP(\omega_2)$ holds in $L(\Gamma,\R)[G_0]$ and thus also $WCC^+_{cof}$.\\
    Fix a surjection $\rho : \R \rightarrow H_{\omega_2}\cap L(\R)$ definable in $\langle L(\R)\cap H_{\omega_2},\in\rangle$.\\
    For a term $\tau \in L(\Gamma,\R)^{\pmax}$ for a function $F :\omega_2^{<\omega} \rightarrow \omega_2$, by weak-homogeneity we can suppose $1 \Vdash  \tau :\omega_2^{<\omega} \rightarrow \omega_2$.
    \\
    \\
    Using determinacy we turn $WCC^+_{cof}$ into a statement about reals in the ground model : 
    \\
    \\
    Consider the following game : players choose pairs of reals $(x_i,y_i)$ such that $\rho (x_i)\in\pmax$ and $\rho(x_{i+1})<\rho(x_i)$ for any $i<\omega$.
    Set $\sigma = \{x_i,y_i\ |\ i<\omega\}$, II wins if for each $p_0\in\pmax$ such that $p_0<\rho(x_i)$ for every i and all $\alpha\in\omega_2$, there is $p\in\pmax$, $p<p_0$ and $Z\in [\omega_2]^\omega$ such that :
    \begin{enumerate}
        \item $\rho[\sigma]\cap\omega_2\subsetneq Z$.
        \item $\rho[\sigma]\cap\omega_1 = Z\cap\omega_1$.
        \item sup $Z>\alpha$.
        \item $p \Vdash \tau[Z^{<\omega}]\subseteq Z$.
    \end{enumerate}
    By $\ADR$ the game is determined and $WCC^+_{cof}$ implies that I does not have a winning strategy :
    \\
    \\
    Otherwise suppose s is a winning strategy for I, we play as II and suppose I plays according to s. Since $WCC^+_{cof}$ holds in $L(\Gamma,\R)[G_0]$ there is $\dot{g}$ in $L(\Gamma,\R)^{\pmax}$ and $q\in\pmax$ such that $\forall X\in[\omega_2]^\omega$
    $q\Vdash \dot{g}[X^{<\omega}]\subseteq X \implies \forall \alpha\in\omega_2$ there is Z that witnesses $WCC^+_{cof}$ for X and $\alpha$.
    Fix $\langle t_i, i<\omega \rangle$ an enumeration of $\omega^{<\omega}$ such that for all i, $t_i\subseteq i$.
    At turn i choose $x_{i+1}$ such that $\rho(x_{i+1})<\rho(x_i),q$ and such that $\rho(x_{i+1})$ fixes the value of $\dot{g}$ on $(\{\rho(y_j)\ |\ j<i\}\cap\omega_2)^{<\omega}$ (we can assume that each $\rho(x_{i})$ is in $\omega_2$, otherwise we just delay the enumeration). Then choose $y_{i+1}$ such that $\rho(x_{i+1})\Vdash\rho(y_{i+1})=\dot{g}(\langle \rho(y_{t_{i}(n)}),n\in dom(t_i) \rangle)$.
    Now for any $p_0$ such that $\forall i<\omega,  p_0<\rho(x_i)$ we have $p_0\Vdash \dot{g}[\rho[\sigma]\cap\omega_2]\subseteq \rho[\sigma]\cap\omega_2$.
    \\
    \item By the choice of $\dot{g}$ we then have that for all $\alpha\in\omega_2$ there is $p<p_0$ and Z that satisfies (1)-(4), which should not be possible since we played against I using his winning strategy s.
    \\
    \\
    The game being determined II must now have a winning strategy, let $h\in L(\Gamma,\R)$ be the function $\R^{<\omega}\rightarrow\R$ defined from this strategy.
    \\
    \\
    $h\in L(\Gamma,\R)$ is then such that for all $\sigma\in P_{\omega_1}(\R)$, if $h[\sigma^{<\omega}]\subseteq \sigma$, then for comeager many filter $g\subseteq \pmax\cap\rho[\sigma]$, for all $p_0<g$ and any $\alpha\in\omega_2$ there is $p<p_0$ and Z that satisfies (1)-(4).
    \\
    \\
    To see this fix $\langle \sigma_i : i<\omega\rangle$ an enumeration of $\sigma$. We play as I against II with his winning strategy s used to define h, assume that at turn i we always choose $y_{i}$ to be $\sigma_i$. Since $\sigma$ is closed by h, as long as we choose a correct $x_{i+1}$ (meaning $\rho(x_{i+1})<\rho(x_i)$) the answer of II will be in $\sigma$. We can consider this as an auxiliary game where both players play decreasing $\pmax$ conditions. Let s' be the wining strategy for II obtained from s.
    \\
    \\
    Suppose $t\in \sigma^{<\omega}$ is a partial play of even length. Set $U_t=\{G \subseteq \pmax\cap\rho[\sigma]\ | \ t\in G^{<\omega} \implies \exists p\in G\ s'(t^\frown p)\in G \}$. Each $U_t$ is open dense hence U=$\bigcap_t U_t$ is a comeager set of filter and any $G\in U$ is as required.
    \\
    \\
    For any name $\tau \in L(\Gamma,\R)^{\pmax}$ for a function $\omega_2^{<\omega} \rightarrow \omega_2$ set $A_\tau =\{(x,a,b)\ |\ \rho(x) \Vdash \tau (a)=b\}$, this way such names can be coded by sets of reals. Define $\Sigma$ as the set of all sets of reals coding such a name.
    \\
    \\
    Fix a surjection $\pi : \Theta^\omega \rightarrow \Gamma$ that is $\Sigma_1$ definable in $L(\Gamma,\R)$ given by theorem \ref{coding}.
    \\
    \\
    For A in $\Sigma$ coding $\tau$
    let $s \in \delta^\omega$ be such that for $B=\pi(s)$, A and $\R-A$ have scales projective in B. 
    \\
    \\
    Notice that the statment "$\sigma \in P_{\omega_1}(\R)$ is such that for comeager many filter $g\subseteq \pmax\cap\rho[\sigma]$, for all $p_0<g$ and any $\alpha\in\omega_2$ there is $p<p_0$ and Z that satisfies (1)-(4)." is first order over $\langle L(\R)\cap H_{\omega_2},\in,A\rangle$, hence the statement "there is h such that for all $\sigma\in P_{\omega_1}(\R)$, if $h[\sigma^{<\omega}]\subseteq \sigma$, then for comeager many filter $g\subseteq \pmax\cap\rho[\sigma]$, for all $p_0<g$ and any $\alpha\in\omega_2$ there is $p<p_0$ and Z that satisfies (1)-(4)." is $\Sigma^2_1(A)$.
    \\
    \\
    By the Solovay basis theorem (theorem \ref{SBT}), there is a $\Delta^2_1(A)$ witness h and such an h is thus ordinal definable from s. Hence for any countable X and $s\in X$, $\R \cap HOD^{L(\Gamma,\R)}(X)$ is closed for such an h.
    \\
    \\
    let $\sigma = \R \cap HOD^{L(\Gamma,\R)}(X)$, using theorem \ref{comeager} we have :
    \\
    \\
    For all $g\subseteq HOD^{L(\Gamma,\R)}(X)\cap \pmax$ that is $HOD^{L(\Gamma,\R)}(X)$-generic and any $p_0\in\pmax$ with $p_0<g$ and $\alpha <\omega_2$, there is $p<p_0$ and $Z\in[\omega_2]^{\omega}$ such that :
    \begin{enumerate}
        \item $\rho[\sigma]\cap\omega_2 \subsetneq Z$.
        \item $\rho[\sigma]\cap\omega_1 = Z\cap\omega_1$.
        \item sup $Z>\alpha$.
        \item $p\Vdash \tau[Z^{<\omega}]\subseteq Z$.
    \end{enumerate}
    ~\\Fix $\delta$ and recall that $\Gamma_\delta=\{ A\subseteq \R\ |\ w(A)<\delta \}$.
    \\
    \\
    Remark : The following differs from the proof given in [6] as we realized that the previous solution does not work, this is a correction given by Woodin.
    \\
    \\
    Pick $t\in \Theta^\omega$ such that $\pi(t)$ is Suslin co-Suslin and wadge above every set of $\Gamma_\delta$
    define $X_\infty=\{j_\infty(s)\ |\ s\in \delta^\omega\ is\ bounded\ in\ \delta\}\cup \{j_\infty(t)\}$ (note that $\R\subseteq X_\infty$).
    \\
    \\
    Using the $\omega_2$ wellordering of $\R$ $\langle x_i,\ i\in \omega_2 \rangle$ given by $\pmax$, define a function $F: \omega_2^{<\omega} \rightarrow\omega_2$ such that if the wadge reduct of $\pi(t)$ by the function coded by $x_\alpha$ is $A_f\in\Sigma$, then for all $\bar{u},\ F(\alpha^\frown\bar{u})=f(\bar{u})$. We can pick $\tau$ a name for such an F ordinal definable from t.
    \\
    \\
    Applying the previous result in $L(\Gamma^\infty,\R^\infty)$ for $j_\infty(t)$ and $X_\infty$ using the embedding $j_\infty$ we have :
    \\
    \\
    For all generic $g\subseteq HOD^{L(\Gamma^\infty,\R^\infty)}(X_\infty)\cap \pmax$ and any $p_0\in j_\infty(\pmax)$ with $p_0<g$ and any $\alpha<j_\infty(\omega_2)$, there is $p<p_0$ and $Z\in j_\infty([\omega_2]^{\omega})$ such that :
    
    \begin{enumerate}
        \item $j_\infty (\rho)[\sigma]\cap \omega_2 \subsetneq Z$.
        \item $j_\infty (\rho)[\sigma]\cap j_\infty(\omega_1^{L(\Gamma,\R)})=Z\cap j_\infty(\omega_1^{L(\Gamma,\R)})$.
        \item sup $Z > \alpha$.
        \item $p\Vdash j_\infty (\tau)[Z^{<\omega}]\subseteq Z$.
    \end{enumerate}
    ~\\
    First notice that by the homogeneity of $Coll(\omega,\R)$, $\R^{L(\Gamma^\infty,\R^\infty)}\cap{HOD^{L(\Gamma^\infty,\R^\infty)}}(X_\infty)=\R$.
    From this we have $j_\infty (\rho)[\sigma]\cap\omega_2=j_\infty (\rho)[\R]\cap\omega_2=j_\infty (\rho[\R])\cap\omega_2=j_\infty[\omega_2^{L(\Gamma,\R)}]$ and $j_\infty (\rho)[\sigma]\cap j_\infty(\omega_1^{L(\Gamma,\R)})=j_\infty[\omega_2^{L(\Gamma,\R)}]\cap j_\infty(\omega_1^{L(\Gamma,\R)})=j_\infty[\omega_1^{L(\Gamma,\R)}]=\omega_1^{L(\Gamma,\R)}$. 
    \\
    \\
    Also $\pmax^{HOD^{L(\Gamma^\infty,\R^\infty)}(X_\infty)}=\pmax^{L(\Gamma,\R)}$ and $G_0\subseteq HOD^{L(\Gamma^\infty,\R^\infty)}(X_\infty)\cap \pmax^{HOD^{L(\Gamma^\infty,\R^\infty)}(X_\infty)} $.
    Again by homogeneity of $Coll(\omega,\R)$, $G_0$ is $HOD^{L(\Gamma^\infty,\R^\infty)}(X_\infty)$-generic.
    \\
    \\
    Finally, by the construction of $\tau$, if $p\Vdash f[Z^{<\omega}]\subseteq Z$ then for any $\alpha\in Z$ if p forces that the wadge reduct coded by $x_\alpha$ is $A_f\in\Sigma$, then $p\Vdash f[Z^{<\omega}]\subseteq Z$. Since $G_0$ is $L(\Gamma,\R)$-generic, for any $x\in \R$ there is $v\in G_0$ and $\alpha \in \omega_2$ such that $v \Vdash x_\alpha=x$.
    Thus if $f \in \Gamma_\delta$, then in $L(\Gamma^\infty,\R^\infty)$ when $q<G_0$, there is $\alpha\in \omega_2^{L(\Gamma,\R)}$ such that $q\Vdash "j_\infty(f)$ is coded by the wadge reduct of $j_\infty(\pi(t))$ by $x_{j_\infty(\alpha)}"$.
    \\
    \\
    Combining everything we have :
    \\
    \\
    (X) For any $p_0\in j_\infty(\pmax)$ with $p_0<G_0$ and any $\alpha<j_\infty(\omega_2)$, there is $p<p_0$ and $Z\in j_\infty([\omega_2]^{\omega})$ such that
    \begin{enumerate}
        \item $j_\infty [\omega_2^{L(\Gamma,\R)}] \subsetneq Z$.
        \item $\omega_1^{L(\Gamma,\R)}=Z\cap j_\infty(\omega_1^{L(\Gamma,\R)})$.
        \item sup $Z > \alpha$.
        \item for any $f$ such that $A_f \in \Gamma_\delta$, $p\Vdash j_\infty (f)[Z^{<\omega}]\subseteq Z$.
    \end{enumerate}
    ~\\
    Next we look at what happens when $SCC^{cof}$ fails : 
    \\
    \\
    Set $M_0$ defined in $N_\delta[G_0]$ :
    $M_0=\{f : \omega_2^{<\omega}\rightarrow\omega_2\ |\ f\in N_\delta[G_0]\} \cup H_{\omega_2}^{N_\delta[G_0]}$.
    \\
    \\
    Set T the set of $X\in [M_0]^\omega$ such that for all $\alpha <\omega_2$ there is $Y\in [M_0]^\omega$ with :
    \begin{enumerate}
        \item $X\cap \omega_2 \subsetneq Y$.
        \item $X\cap\omega_1 = Y\cap \omega_1$.
        \item sup $Y > \alpha$.
        \item for each $f\in X, f : \omega_2^{<\omega} \rightarrow \omega_2$, Y is closed by $f$.
    \end{enumerate}
    ~\\
    By lemma \ref{scc} if T is a club then $SCC^{cof}$ holds, suppose then that T is not a club so that $S=[M_0]^{\omega}-$T is stationary.
    \\
    \\
    Let $H_0\subseteq Coll(\omega_3,P(\omega_2))^{N_\delta[G_0]}$ be $N_\delta[G_0]$-generic such that $H_0\in L(\Gamma,\R)[G_\infty]$. Remark then that $N_\delta[G_0][H_0]\models ZFC$.
    \\
    \\
    Since $\kappa=\Theta_{\delta+1}$ is Woodin in $HOD^{L(\Gamma,\R)}$ it follows that $\kappa$ is Woodin in $N_\delta[G_0][H_0]$. We can then apply the stationary tower forcing $\mathbb{Q}_{<\kappa}$:
    Choose $G_{S}\subseteq (\mathbb{Q}_{<\kappa})^{N_\delta[G_0][H_0]}$ $N_\delta[G_0][H_0]$-generic with $S\in G_S$ such that $G_S\in L(\Gamma,\R)[G_\infty]$.
    \\
    \\
    Choose also $g_0\subseteq Coll(\omega,P(\delta)\cap N_\delta[G_0][H_0])$ an $N_\delta[G_0][H_0]$-generic filter such that $g_0 \in N_\delta[G_0][H_0][G_S]$. 
    \\
    \\
    For $j_S : N_\delta[G_0][H_0] \rightarrow N^S\subseteq N_\delta[G_0][H_0][G_S]$ the stationary tower embedding defined from $G_S$, since $M_0=\cup S$ and $S\in G_S$ we have $j_S[M_0]\in j_S(S)$, it follows that in $j_S[N_\delta[G_0]]$ :
    \\
    \\
    (Y) There is $p_0\in j_S(G_0)$ such that $p_0<G_0$ and such that there is an $\alpha<j_S (\omega_2)$ such that for all $Z\in j_S(P_{\omega_1}(\omega_2))$, if :
    \begin{enumerate}
        \item $j_S[\omega_2^{L(\Gamma,\R)}]\subsetneq Z$.
        \item $\omega_1^{L(\Gamma,\R)}= Z\cap j_S(\omega_1)^{L(\Gamma,\R)}$.
        \item sup $Z> \alpha$.
    \end{enumerate}
    then there is $p\in j_S(G_0)$ and $\tau$ with $A_\tau \in \Gamma_\delta$ such that
    \begin{enumerate}
        \item $p \Vdash j_S(\tau)[Z^{<\omega}]\not \subseteq Z$.
        \item $p<p_0$.
    \end{enumerate}
    ~\\
    To derive a contradiction we show how the statements (X) and (Y) can be coded by sets of reals (in $N_S$ and $L(\Gamma,\R)[G_\infty]$ respectively) using the same formula.
    \\
    \\
    For any $A\in \Gamma_\delta$ there is a weakly homogeneous tree $T_A\in N_\delta$ such that $A=p[T_A]$, pose $T^\infty_A=j_\infty(T_A)=j_\mu(T_A)\in N_\delta$ and $A^\infty=p[T^\infty_A]$ computed in $L(\Gamma,\R)[G_\infty]$.
    \\
    \\
    We have $A^\infty=j_\infty(A)$ and $j_S(A)=j_\infty(A)\cap N_S$:
    We show that $j_S(A)=p[T^\infty_A]^{N_{S}}$.
    \\
    \\
    First observe that in $N_\delta$ $A=p[T_A]=p[T_A^\infty]$: 
    \\
    
        Since $T_A\subseteq T_A^\infty$ we have $p[T_A]\subseteq p[T_A^\infty]$.
    \\
    
    Suppose $x \in p[T_A^\infty]$ and $x \in \R^{N_\delta}$, then $x=j_\infty(x)\in p[T_A^\infty]=j_\infty(A)$ and thus $x\in A = p[T_A]$.
    \\
    \\
    Hence in $N_\delta$, $A=p[T^\infty_A]$, by theorem \ref{tower} we have $j_S(A)=p[T^\infty_A]^{N_S}$.
    \\
    \\
    Set $g_0 \subseteq Coll(\omega,P(\delta)\cap N_\delta[G_0][H_0])$ be an $N_\delta[G_0][H_0])$-generic filter in $N_\delta[G_0][H_0][G_S]$. Note that $g_0\in N_S$.
    \\
    \\
    Using the enumeration $\langle A_i,i<\omega\rangle $ of $\Gamma_\delta$ given by $g_0$  define B as the set of $x\in \R^\infty$ coding $\langle x_i,i < \omega\rangle$ such that $x_i\in p[T^\infty_{A_i}]$, merging the trees $\langle T^\infty_{A_i},i<\omega\rangle $ we obtain a tree $T^\infty$ in $N_\delta[G_0][H_0][G_S]$ such that $B=p[T^\infty]$.
    \\
    \\
    Since every tree $T^\infty_A$ is $<\Theta$-weakly homogeneous, the tree $T^\infty$ is $<\Theta$-weakly homogeneous in $N_\delta[G_0][H_0][g_0]$.
    \\
    \\
    Claim : $\langle H_{\omega_1}^{N_S},B\cap N^S,\in \rangle \prec \langle H_{\omega_1}^{L(\Gamma,\R)[G_\infty]},B,\in \rangle$
    (see [6])
    \\
    \\
    Using a projective prewellordering A of $\R$ (thus $A\in \Gamma_\delta$) to represent $\omega_2$ (meaning $\omega_2=\{r_A(x)\ |\ x\in \R \}$, where $r_A$ is the rank function of A), $j_S[\omega_2^{L(\Gamma,\R)}]$ and $j_\infty[\omega_2^{L(\Gamma,\R)}]$ are both represented by $\R$ in $N_S$ and $L(\Gamma,\R)[G_\infty]$ respectively. ($j_S[\omega_2^{L(\Gamma,\R)}]=\{ r_{j_S(A)}(x)\ |\ x\in \R \}$ and $j_\infty[\omega_2^{L(\Gamma,\R)}]=\{ r_{j_\infty(A)}(x)\ |\ x\in \R \}$).
    \\
    \\
    Since $G_0$ and $\R$ are countable in $N_S$ and $L(\Gamma,\R)[G_\infty]$, (X) is expressible in $\langle H_{\omega_1}^{N_S},B\cap N^S,\in \rangle$ with a formula with parameters $G_0 ,\R$ (actually $G_0$ is enough since $\R$ is the set of reals occurring in conditions of $G_0$)  and (Y) is expressible by the negation of the same formula in $ \langle H_{\omega_1}^{L(\Gamma,\R)[G_\infty]},B,\in \rangle$ we have a contradiction.
    \\
    \\
    Consequently S must be nonstationary and $SCC^{cof}$ holds in $N_\delta[G_0]$.
       
\end{proof}

The main interest in this stronger version of the conjecture is the following theorem from Shelah. 

\begin{theorem}(Shelah)
    $SCC^{cof}$ is equivalent to Namba forcing being semi-proper.
\end{theorem}

Here we apply the result of the theorem to get a model of "ZFC + Namba forcing is semiproper + $\Theta^{UB}=\omega_3$". Using an appropriate $\delta$ in the $\Theta$-sequence of $L(\Gamma,\R)$ we will have a model. 

\begin{theorem}
    Suppose $\Gamma$ is a pointclass closed under continuous pre-images such that
    \\
    \\
    \centerline{$L(\Gamma,\R)\models \ADR + \Theta\ is\ regular$}
    \\
    \item Let $\langle \Theta_\alpha : \alpha < \Omega \rangle$ be the $\Theta-sequence$ of $L(\Gamma,\R)$, pose for each $\delta<\Omega$,
    
    \centerline{$\Gamma_\delta = \{A\subseteq \R\ |\ w(A)< \Theta_\delta\}$}
    
    \item set
    
    \centerline{$N_{\delta}=HOD(\Gamma_\delta)$}
    
    \item if $\delta<\Omega$ is such that 
    \begin{enumerate}
        \item $\delta=\Theta_\delta$.
        \item $N_\delta\models \delta\ is\ regular$.
    \end{enumerate}
    then for $G_0 \subseteq \pmax$ $L(\Gamma,\R)$-generic and $H_0$ is $L(\Gamma,\R)[G_0]$-generic for $coll(\omega_3,P(\omega_2))^{N_\delta[G_0]}$, then $V_{\theta_{\delta +1}}^{N_\delta[G_0][H_0]}\models ZFC+SCC^{cof}+ \Theta^{UB}=\omega_3$.
\end{theorem}

\begin{proof}
    By theorem \ref{scc pmax} we have $V_{\theta_{\delta+1}}^{N_\delta[G_0][H_0]}\models ZFC+SCC^{cof}$ and by theorem \ref{ub} all sets of reals in $N_\delta$ are universally baire in $V_{\theta_{\delta+1}}^{N_\delta}$. Since $\pmax$ forces $|\R|=\omega_2$ we have $V_{\theta_{\delta+1}}^{N_\delta[G_0][H_0]}\models \omega_3\geq \Theta^{UB}\geq (\Theta^{UB})^{V_{\theta_{\delta+1}}^{N_\delta}}=\Theta^{V_{\theta_{\delta+1}}^{N_\delta}}=\omega_3$.
\end{proof}

\begin{corollary}
    The theory "ZFC + Namba forcing is semiproper + $\Theta^{UB}=\omega_3$" is consistent.
\end{corollary}

The main theorem shows that $SCC^{cof}$ holds in a specific case of a $\pmax$ extension. The next theorem isolate a first order and more general characterization of what is needed to get the conjecture.

\begin{theorem}
    Suppose there exist two Woodin cardinals $\lambda>\delta$ and that there is a class $\Gamma\subseteq P(\R)$ such that $L(P(\R))=L(\Gamma,\R)[G]$ where $G\subseteq \pmax$ is generic, $L(\Gamma,\R)\models \ADR + \Theta$ is regular, $\Gamma=P(\R)\cap L(\Gamma,\R)$ and for some $\sigma>\lambda$ every set in $\Gamma$ is $\sigma-$weakly homogenous. Then $SCC^{cof}$ holds. 
\end{theorem}

\begin{proof}
    Fix K $V-generic$ for $Q_{<\lambda}$ and H $V-$generic for $\Q_{<\delta}^{\omega_2}$ in V[K] with $\{X\in [\omega_3]^{\omega_2}\ |\ X\cap\omega_3\in\omega_3\}\in H$ and $j_H : V \rightarrow N \subseteq V[H]$ the associated embedding. Note that $L(P(\R))=L(\Gamma,\R)[G]$ implies that $|\R|=\omega_2$ and $|\Gamma|\geq \omega_3$.
    \\
    \\
    Hence we have $j_H(\R)=\R$ and $\Gamma \subsetneq j_H(\Gamma)$ since $crit(j_H)=\omega_3$, this also implies $j_H(A)=A$ for any $A\subseteq \R$. We then have $j_H : L(\Gamma,\R)[G] \rightarrow L(j_H(\Gamma),\R)[G]$ elementary. Further since $j_H(\R)=\R$, $\Gamma$ is wadge closed in $j_H(\Gamma)$ and $\Gamma = \{A\subseteq \R\ |\ w(A)<\Theta^{\Gamma}\}$ where $\Theta^{\Gamma}$ is the $\Theta$ of $L(\Gamma,\R)$.
    \\
    \\
   Fix $G_\infty \subseteq coll(\omega,\R)$ $V[H]$-generic in V[K] such that $G\in L(j_H(\Gamma),\R)[G_\infty]$ and $j_\infty : L( j_H(\Gamma),\R) \rightarrow L(j_H(\Gamma)^\infty,\R^\infty)$ the associated embedding.
   \\
    \\
    By elementarity of $j_H$ we get $N\models L(P(\R))=L(j_H(\Gamma),\R)[G]$, since names for reals in an extension by $coll(\omega,\R)$ can be coded by sets of reals and $G\in L(j_H(\Gamma),\R)[G_\infty] $ we get $\R^{N[G_\infty]}=\R^{L(j_H(\Gamma),\R)[G_\infty]}$. Again since $N^{\omega_2}\subseteq N$ in V[H] and $|\R|=\omega_2$ we get $P(\R)^N=P(\R)^{V[H]}$ and then $\R^{N[G_\infty]}=\R^{V[H][G_\infty]}$.
    \\
    \\
   Applying the first part of the proof of the main theorem in $L(j_H(\Gamma),\R)$ we easily get : 
   \\
    
    For sequence s where $\pi(s)$ codes $\tau$ ($\pi : \Theta^\omega \rightarrow \Gamma$ is a surjection $\Sigma_1$ definable in $L(\Gamma,\R)$ given by theorem \ref{coding}.) and any countable X such that $s\in X$ and $\sigma = \R \cap HOD^{L(j_H(\Gamma),\R)}(X)$:
    \\
   \\
    For all $g\subseteq HOD^{L(j_H(\Gamma),\R)}(X)\cap \pmax$ that is $HOD^{L(j_H(\Gamma),\R)}(X)$-generic and any $p_0\in\pmax$ with $p_0<g$ and $\alpha <\omega_2$, there is $p<p_0$ and $Z\in[\omega_2]^{\omega}$ such that :
    \begin{enumerate}
        \item $\rho[\sigma]\cap\omega_2 \subsetneq Z$.
        \item $\rho[\sigma]\cap\omega_1 = Z\cap\omega_1$.
        \item sup $Z>\alpha$.
        \item $p\Vdash \tau[Z^{<\omega}]\subseteq Z$.
    \end{enumerate}
    where $\rho : \R \rightarrow H_{\omega_2}\cap L(\R)$ is a surjection definable in $\langle L(\R)\cap H_{\omega_2},\in\rangle$.
    \\
    \\
     We apply the same process as with the main theorem but for the full $\Gamma$ using the fact that $\Gamma$ becomes wadge bound in $L(j_H(\Gamma),\R)$. Pick $t\in \Theta^\omega$ such that $\pi(t)$ is Suslin co-Suslin and wadge above every set of $\Gamma$
    define $X_\infty=\{j_\infty(s)\ |\ s\in {\Theta^\Gamma}^\omega\ is\ bounded\ in\ \Theta^\Gamma\}\cup \{j_\infty(t)\}$ (note that $\R\subseteq X_\infty$).
    \\
    \\
    Using the $\omega_2$ wellordering of $\R$ $\langle x_i,\ i\in \omega_2 \rangle$ given by $\pmax$, define a function $F: \omega_2^{<\omega} \rightarrow\omega_2$ such that if the wadge reduct of $\pi(t)$ by the function coded by $x_\alpha$ is $A_f\in\Gamma$, then for all $\bar{u},\ F(\alpha^\frown\bar{u})=f(\bar{u})$. We can pick $\tau$ a name for such an F ordinal definable from t.
    \\
    \\
    Applying the previous result in $L(j_H(\Gamma)^\infty,\R^\infty)$ for $j_\infty(t)$ and $X_\infty$ using the embedding $j_\infty$ we have :
    \\
    \\
    For all generic $g\subseteq HOD^{L(j_H(\Gamma)^\infty,\R^\infty)}(X_\infty)\cap \pmax$ and any $p_0\in j_\infty(\pmax)$ with $p_0<g$ and any $\alpha<j_\infty(\omega_2)$, there is $p<p_0$ and $Z\in j_\infty([\omega_2]^{\omega})$ such that :
    
    \begin{enumerate}
        \item $j_\infty (\rho)[\sigma]\cap \omega_2 \subsetneq Z$.
        \item $j_\infty (\rho)[\sigma]\cap j_\infty(\omega_1^{L(j_H(\Gamma),\R)})=Z\cap j_\infty(\omega_1^{L(j_H(\Gamma),\R)})$.
        \item sup $Z > \alpha$.
        \item $p\Vdash j_\infty (\tau)[Z^{<\omega}]\subseteq Z$.
    \end{enumerate}
    \bigskip

    Pick a set $U\subseteq \R$ in $L(\Gamma,\R)$ that codes a prewellordering of length $\omega_2$ and define $l(U)$ as the length of this ordering (hence in $L(\Gamma,\R)$ $l(U)=\omega_2$), with the same arguments as the main theorem we get in $L(j_H(\Gamma)^\infty,\R^\infty)$ :
    \\
    \\
    (X) For any $p_0\in j_\infty(\pmax)$ with $p_0<G$ and any $\alpha\in l(j_\infty(U))$, there is $p<p_0$ and $Z\in j_\infty(P_{\omega_1}(\omega_2)^{L(j_H(\Gamma),\R)})$ such that
    \begin{enumerate}
        \item $j_\infty [\omega_2^{L(j_H(\Gamma),\R)}] \subsetneq Z$.
        \item $\omega_1^{L(j_H(\Gamma),\R)}=Z\cap j_\infty(\omega_1^{L(j_H(\Gamma),\R)})$.
        \item sup $Z > \alpha$.
        \item for any $f$ coded by $A_f \in \Gamma$, $p\Vdash j_\infty (f)[Z^{<\omega}]\subseteq Z$.
    \end{enumerate}
    ~\\
    Since $L(j_H(\Gamma)^\infty,\R^\infty)$ is closed by countable sequences (X) holds in $L(j_H(\Gamma),\R)[G_\infty]$ and then also in $N[H][G_\infty]$ and then in $V[H][G_\infty]$ since those models have the same reals.
    \\
    \\
    Finally for any $A\in \Gamma$ choose a tree $T_A\in L(\Gamma,\R)$ such that in V $A=p[T_A]$ (possible by $\ADR$), then in V[H] $A=p[j_H(T_A)]$ since $j_H(A)=A$. First notice that $T^\infty_A=j_\infty(j_H(T_A))=j_\mu(j_H(T_A))\in L(j_H(\Gamma),\R)$, then $T^\infty_A$ represents A in $L(j_H(\Gamma),\R)$ and $j_\infty(A)=p[T^\infty_A]$ is just the generic interpretation of A in $V[H][G_\infty]$.
    \\
    \\
    Next in $V[H]$ the cardinality of $\Gamma$ is collapsed to $\omega_2=|\R|$, since in $V[H][G_\infty]$ $\omega_2$ is collapsed to $\omega$ there exists a countable enumeration of $\Gamma$. We can use it to merge $\sigma$-weakly homogeneous trees representing elements of $\Gamma$ into one $\sigma$-weakly homogeneous tree T and define B=p[T].
    \\
    \\
    Back in V set $M_0=\{f : \omega_2^{<\omega}\rightarrow\omega_2\ |\ f\in V\} \cup H_{\omega_2}$.
    \\
    \\
    Let T be the set of $X\in [M_0]^\omega$ such that for all $\alpha <\omega_2$ there is $Y\in [M_0]^\omega$ with :
    \begin{enumerate}
        \item $X\cap \omega_2 \subsetneq Y$.
        \item $X\cap\omega_1 = Y\cap \omega_1$.
        \item sup $Y > \alpha$.
        \item for each $f\in X, f : \omega_2^{<\omega} \rightarrow \omega_2$, Y is closed by $f$.
    \end{enumerate}
    ~\\
    If $SCC^{cof}$ fails then $S=[M_0]^\omega-$T is stationary, assume this is the case. Recall that $K\subseteq Q_{<\lambda}$ is such that H,$G_\infty\in V[K]$, note $j_K :V \rightarrow M$ the associated generic embedding.
    \\
    \\
    Then in the generic ultrapower M given by K, we have $j_K[M_0]\in j_K(S)$. This implies in M :
    \\
    \\
    There is $\alpha < \omega_2$ such that for any $Z\subseteq \omega_2$, if :
    \begin{enumerate}
        \item $j_K[\omega_2] \subsetneq Z$.
        \item $\omega_1=Z\cap j_K(\omega_1)$.
        \item sup $Z > \alpha$.
    \end{enumerate}
    Then there is f such that $j_K(f)[Z^{<\omega}]\not \subseteq Z$.
    \\
    \\
    But since V verifies $L(P(\R))=L(\Gamma,\R)[G]$ this implies in M : 
    \\
    \\
    (Y) There is $p_0\in j_K(G)$ such that $p_0<G$ and such that there is an $\alpha<j_K (\omega_2^{L(\Gamma,\R)})$ such that for all $Z\in j_K(P_{\omega_1}(\omega_2)^{L(\Gamma,\R)})$, if :
    \begin{enumerate}
        \item $j_K[\omega_2^{L(\Gamma,\R)}]\subsetneq Z$.
        \item $\omega_1^{L(\Gamma,\R)}= Z\cap j_K(\omega_1)^{L(\Gamma,\R)}$.
        \item sup $Z> \alpha$.
    \end{enumerate}
    then there is $p\in j_K(G)$ and $\tau$ coded by $A_\tau \in \Gamma$ such that
    \begin{enumerate}
        \item $p \Vdash j_K(\tau)[Z^{<\omega}]\not \subseteq Z$.
        \item $p<p_0$.
    \end{enumerate}
     
    To conclude, (X) can be expressed in $\langle{H_{\omega_1}}^{V[H][G_\infty]},B,\in \rangle$ by a formula with parameter G and (Y) in $\langle H_{\omega_1}^{V[K]},B^{*},\in \rangle$ by the negation of the same formula where quantification on $\omega_2$ is done using $j_K(U)$ and $B^{*}$ is the projection of T in V[K] (since the stationary tower sends a homogeneous set to its generic interpretation and $H_{\omega_1}^{V[K]}=H_{\omega_1}^{M}$).
    \\
    \\
    Since V[K] is a generic extension of $V[H][G_\infty]$ and B is weakly homogeneous we have
    $\langle{H_{\omega_1}}^{V[H][G_\infty]},B,\in \rangle \prec \langle H_{\omega_1}^{V[K]},B^{*},\in \rangle$, this is contradictory and thus implies that $SCC^{cof}$ must hold in V.

\end{proof}

Using this new result we can get an interesting relation with Woodin $(\ast)_{UB}$. Recall first its definition, we assume a ZFC context. We note $\Gamma^\infty$ for the class of universally baire sets of reals.

\begin{definition}
    The axiom $(\ast)_{UB}$ is the conjunction of the following statements :
    \begin{enumerate}
        \item There is a proper class of Woodin cardinals.
        \item $\Gamma^\infty=P(\R)\cap L(\Gamma^\infty,\R)$.
        \item $L(\Gamma^\infty,\R)\models \AD^+$.
        \item $L(P(\R))=L(\Gamma^\infty,\R)[G]$ for some filter $G\subseteq \pmax$.
    \end{enumerate}
\end{definition}

First remark a few things in the $(\ast)_{UB}$ context. Since there is proper class many Woodin cardinals, for any ordinal $\lambda$ there is a cardinal $\sigma>\lambda$ such that every set in $\Gamma^\infty$ is $\sigma$-weakly homogeneous. Also the proper class of Woodins imply that every set in $\Gamma^\infty$ has a scale in $\Gamma^\infty$, hence since $\Gamma^\infty=P(\R)\cap L(\Gamma^\infty,\R)$ every set in $L(\Gamma^\infty,\R)$ is Suslin in $L(\Gamma^\infty,\R)$, together with $\AD^+$ we get $L(\Gamma^\infty,\R)\models \ADR$. Finally since choice holds and $L(P(\R))=L(\Gamma^\infty,\R)[G]$ we have that $\omega_2-\DC$ holds in $L(\Gamma^\infty,\R)[G]$, this is equivalent to the $\Theta$ of $L(\Gamma^\infty,\R)$ being regular. We then immediately get : 

\begin{theorem}
    $(\ast)_{UB}$ implies $SCC^{cof}$.
\end{theorem}
~\\
$Acknowledgments$. The author thanks Grigor Sargsyan for his supervision and guidance throughout this research, and Hugh Woodin for his valuable input regarding the correction of an earlier proof of the main theorem.

\bibliographystyle{plain}
\bibliography{References}

\nocite{Wo10}
\nocite{coxcc}
\nocite{ADRUB}
\nocite{SargMu}
\nocite{ESSealing}
\nocite{La04}

\begin{figure}[h!]
    \centering
    \includegraphics[width=0.5\linewidth]{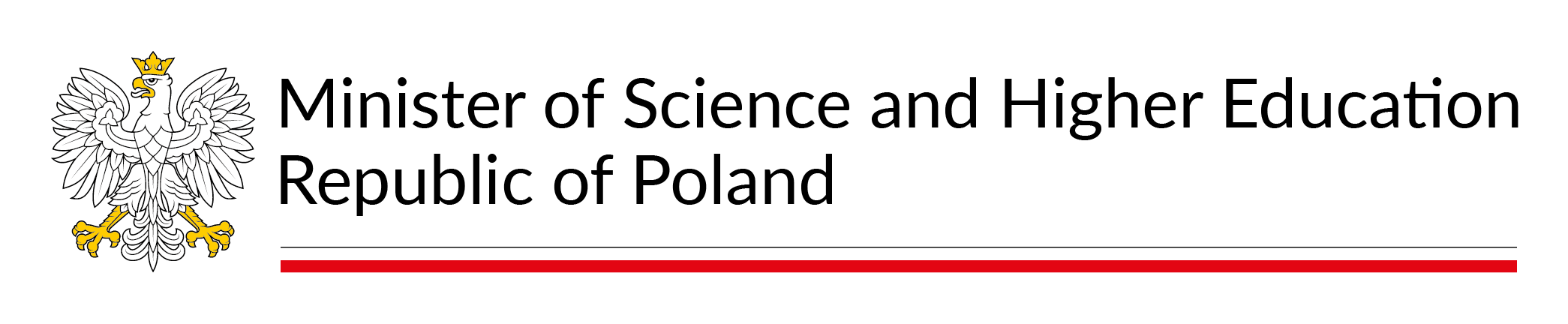}
    \label{fig:placeholder}
\end{figure}

\begin{figure}[h!]
    \centering
    \includegraphics[width=0.5\linewidth]{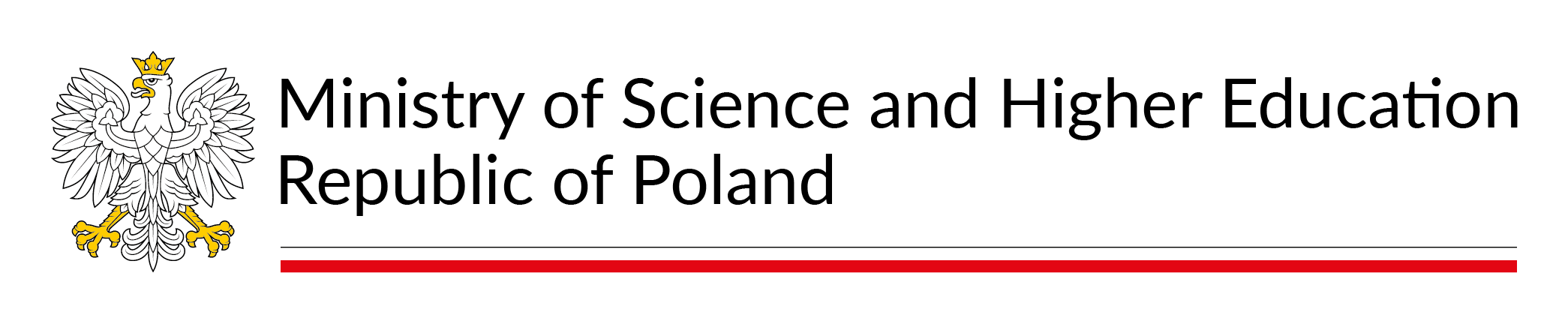}
    \label{fig:placeholder}
\end{figure}

\end{document}